\documentclass[12pt]{amsart}
\usepackage{amssymb}
\usepackage[mathscr]{eucal}
\usepackage{epsf}
\usepackage{epsfig}
\usepackage{color}
\usepackage{graphicx}
\DeclareGraphicsExtensions{.pstex,.eps,.epsi,.ps}

\makeatletter
\def\LaTeX{\leavevmode L\raise.42ex
    \hbox{\kern-.3em\size{\sf@size}{0pt}\selectfont A}\kern-.15em\TeX}
\makeatother

\newcommand{\BibTeX}{{\rm B\kern-.05em{\sc i\kern-.025emb}\kern-.08em\TeX}}

\newtheorem{thm}{Theorem}[section]
\newtheorem{prop}[thm]{Proposition}
\newtheorem{prop-def}[thm]{Proposition-Definition}
\newtheorem{lem}[thm]{Lemma}
\newtheorem{cor}[thm]{Corollary}

\theoremstyle{definition}

\theoremstyle{remark}

\theoremstyle{remark}

\makeatletter
\def\@currentlabel{2.1}\label{e:dispaa}
\def\@currentlabel{2.21}\label{e:dispau}
\def\@currentlabel{2.22}\label{e:dispav}
\def\@currentlabel{2.23}\label{e:dispaw}
\def\@currentlabel{2.24}\label{e:dispax}
\def\theequation{\thesection.\@arabic\c@equation}
\makeatother

\makeatletter
\def\alphenumi{%
  \def\theenumi{\alph{enumi}}%
  \def\p@enumi{\theenumi}%
  \def\labelenumi{(\@alph\c@enumi)}}
\makeatother

\newcommand{\e}{\epsilon}

\newcommand{\Real}{\mathbb{R}}

\newcommand{\ddt}{\frac{d}{dt}}
\newcommand{\dd}[1]{\frac{d}{d#1}}

\newcommand{\flow}[2]{\overrightarrow{\exp}\int_{#1}^{#2}}
\begin{document}

\title[Continuity of Optimal Control Cost and Weak KAM Theory]
{Continuity of Optimal Control Costs and its application to Weak KAM Theory}

\author{Andrei Agrachev}
\email{agrachev@sissa.it}
\address{International School for Advanced Studies
via Beirut 4 - 34014, Trieste, Italy and Steklov Mathematical Institute, ul. Gubkina 8,
Moscow, 119991 Russia}

\author{Paul W.Y. Lee}
\email{plee@math.berkeley.edu}
\address{Department of Mathematics, University of California at Berkeley, 970 Evans Hall \#3840 Berkeley, CA 94720-3840 USA}

\thanks{The first author was supported by PRIN and the second author was supported by the NSERC postdoctoral fellowship.}

\begin{abstract}
We prove continuity of certain cost functions arising from optimal control of affine control systems. We give sharp sufficient conditions for this continuity. As an application, we prove a version of weak KAM theorem and consider the Aubry-Mather problems corresponding to these systems.
\end{abstract}

\maketitle


\

\section{Introduction}

Integrability of Hamiltonian systems has been a subject of  considerable interest for several decades. One way to understand the dynamics of such systems is to find a family of smooth solutions, called generating functions, to the time-independent Hamilton-Jacobi equation. These generating functions define symplectic transformations which transform the given completely integrable Hamiltonian system to a much simpler one that are easily solvable. 

On the contrary, if the Hamiltonian system is not completely integrable, then it is natural to ask whether one can solve the Hamilton-Jacobi equation in certain weak sense. This is accomplished in, what is known as, the weak KAM theorem under certain assumptions on the Hamiltonian. More precisely, let $L:TM\to\Real$ be a Lagrangian defined on the tangent bundle $TM$ of a compact manifold $M$ which satisfies the following conditions:
\begin{enumerate}
\item the restriction of the Lagrangian $L$ to each tangent space has positive definite Hessian, 
\item $L(x,v)\geq C|v|^2+K$ for some Riemannian metric $|\cdot|$ and some constants $K,C>0$. 
\end{enumerate}

Let $H:T^*M\to\Real$ be the corresponding Hamiltonian defined by the Legendre transform:
\[
H(x,\alpha)=\sup_{v\in T_xM}[\alpha(v)-L(x,v)]. 
\]

The following is the weak KAM theorem mentioned above. It was first proven in \cite{LiPaVa} when $M$ is a torus and was extended to all compact manifolds in \cite{Fa1} (see also \cite{Go} for a version related to vakonomic mechanics). 

\begin{thm}
Under the above assumptions, there exists a unique constant $h$ such that the Hamilton-Jacobi equation 
\begin{equation}\label{sHJB}
H(x,df_x)=-h,
\end{equation}
has a viscosity solution.
\end{thm}

In order to give the definition of viscosity solution, we first recall the concepts of sub- and super- differentials. If $f$ is a continuous function on a manifold $M$, then the sub-differential $d^-f_x$ of the function $f$ at a point $x$ is the subset of the cotangent space $T^*_xM$ defined by the following: a co-vector $p$ in the cotangent space $T^*_xM$ is contained in the sub-differential $d^-f_x$ of $f$ at $x$ if there exists a smooth function $g$ defined in a neighborhood $O$ of $x$ such that $dg_x=p$ and $g$ touches $f$ from above. By $g$ touching from above, we mean that $f(x)=g(x)$ and $f(y)\leq g(y)$ for all $y$ in the set $O$. The super-differential $d^+f$ of $f$ is defined in a similar way with the function $g$ touching from below instead. Let $G:\Real\times T^*M\to\Real$ be a continuous function, then a continuous function $f$ is called a sub-solution to the equation $G(f(x),x,p)=0$ if for each $p$ in the sub-differential $d^-f_x$,
\[
G(f(x),x,p)\leq 0.
\]
Similarly, $f$ is a super-solution if for each $p$ in the super-differential $d^+f_x$,
\[
G(f(x),x,p)\geq 0.
\]
If $f$ is both a super and a sub-solution, then it is called a viscosity solution (see \cite{CaSi} for various different characterizations of the sub-differential and viscosity solution).

In this paper, we study weak KAM theorem corresponding to Hamiltonians which arise from certain optimal control problems. More precisely, let $X_0,X_1,...,X_n$ be smooth vector fields on a compact manifold $M$ of dimension $m$ and consider the following family of ODEs, called control-affine system:
\begin{equation}\label{controlaff}
 \dot x(t)=F(x(t),u(t)):=X_0(x(t))+\sum_{i=1}^nu_i(t)X_i(x(t)),
\end{equation}
where $u(\cdot):=(u_1(\cdot),...,u_n(\cdot)):[0,T]\to \Real^n$ are essentially bounded measurable functions, called controls, and solutions to (\ref{controlaff}) are Lipschitz curves in $M$, called admissible paths.

Let $L:M\times\Real^n\to \Real$ be a smooth function, called Lagrangian. The optimal control cost $c_T$ corresponding to the above control affine system (\ref{controlaff}) and Lagrangian $L$ is the following function:
\begin{equation}\label{controlcost}
 c_T(x,y)=\inf\int_0^TL(x(t),u(t))dt,
\end{equation}
where the infimum is taken over all pairs $(x(\cdot),u(\cdot))$ which satisfies the affine control system (\ref{controlaff}) and the boundary conditions $x(0)=x$ and $x(T)=y$.

Since there may exist points which are not connected by any admissible path, the above cost function is not always well-defined without additional assumptions. We recall that a family of vector fields $\{X_1,...,X_n\}$ is said to be $k$-generating if the vector fields $X_i$ and their iterated Lie brackets up to $k-1$ order spanned each tangent space in $TM$. More precisely, the following holds for each point $x$ in the manifold $M$
\[
T_xM=\text{span} \{[X_{i_1},[X_{i_2},...,[X_{i_{l-1}},X_{i_l}]]](x)\mid 1\leq i_j\leq n,1\leq l\leq k\}.
\]
The family $\{X_1,...,X_n\}$ is bracket generating if it is $k$-generating for some $k$. If we assume that the family  $\{X_1,...,X_n\}$ is bracket generating, then any two points can be connected by an admissible path \cite{BrLo}. Therefore, under this assumption, the cost $c_T$ in (\ref{controlcost}) is well-defined for any $T>0$ and any points $x$,$y$ on the manifold $M$. 

In this paper, we prove continuity of the optimal control cost $c_T$ under some growth and convexity conditions on the Lagrangian $L$ (see Theorem \ref{morecontinuity}). A simple useful corollary of the general continuity result is as follows:

\begin{thm}\label{continuity}
Assume that the Lagrangian $L$ and the vector fields $X_1,...,X_n$ satisfy the following conditions:
\begin{enumerate}
\item $C_1|u|^q+K_1\leq L(x,u)\leq C_2|u|^2+K_2$,
\item $|\frac{\partial L(x,u)}{\partial x}|\leq C_3|u|^2$,
\item  the Hessian of $L$ in the $u$ variable is positive definite, and
\item $\{X_1,...,X_n\}$ is $3$-generating
\end{enumerate}
for some constants $C_1,C_2,C_3,K_1,K_2>0$ and some constant $q>1$. Then the cost function $(t,x,y)\mapsto c_t(x,y)$ defined in (\ref{controlcost}) is continuous.
\end{thm}

As an application, we prove a version of the weak KAM theorem corresponding to the above optimal control cost $c$. More precisely, let $H:T^*M\to\Real$ be the Hamiltonian function defined by
\begin{equation}\label{Hamiltonian}
 H(x,\alpha_x)=\sup_{u\in U}\left[\alpha_x(F(x,u))-L(x,u)\right]
\end{equation}
Note that the Hamiltonian $H$ is, in general, neither fiberwise strictly convex nor coercive, which are basic assumptions on the classical weak KAM theory (see \cite{Fa2}). 

\begin{thm}(Weak KAM Theorem)\label{wKAM}
If we make the same assumptions as in Theorem \ref{continuity}, then there exists a unique constant $h$ such that the Hamilton-Jacobi equation (\ref{sHJB}) has a viscosity solution.
\end{thm}

The structure of this paper is as follows. In Section 2, we give a counter example showing that the 3-generating condition in Theorem \ref{continuity} is essential. Section 3 and Section 4 are devoted to the proof of Theorem \ref{continuity} and \ref{wKAM}, respectively. In Section 5, we study a generalization of the Aubry-Mather problem to the present setting.

\

\section{Example}

Assume that $M$ is two-dimensional and the control system has the
form:
\[
\dot x_1=u_1, \quad \dot x_2=x_1^2+u_2x_1^k,
\]
in some local coordinate chart.
The family of vector fields $X_1(x_1,x_2)=(1,0)$ and $X_2(x_1,x_2)=(0,x_1^k)$ is $(k+1)$-generating but not $k$-generating. In this section, we show that the cost function $c_1$ corresponding to the Lagrangian $L(x,u)=\frac{1}{2}(u_1^2+u_2^2)$ is not continuous if $k\geq 3$. This shows that the 3-generating assumption in Theorem \ref{continuity} is essential. More precisely,

\begin{prop}
Assume that $k\geq 3$. Then the cost function $c_1$ corresponding to the above control system and Lagrangian satisfies
\[
c_1((0,w),(0,w))=0,\quad c_1((0,w),(0,z))\geq K
\]
for some constant $K>0$, all $w$, and all $z<0$.
In particular, the cost function $c_1$ is not continuous. 
\end{prop}

\begin{proof}
According to the result in \cite{CaRi}, the cost function $c_1$ is much better than continuous (in fact semiconcave) at $(x,y)$ if the points $x$ and $y$ are not connected by abnormal minimizers (see \cite{AgSa} or below for the definitions of normal and abnormal minimizers). Therefore, let us apply Pontryagin maximum principle and find candidates for which the cost function $c_1$ is not continuous. 

Let $H_u^\nu$ be the Hamiltonian function defined by
\[
H_u^\nu(x,p)=p(F(x,u))+\nu L(x,u).
\]

By applying Pontryagin maximum principle (see, for instance, \cite{AgSa}), any minimizer $(x(\cdot),u(\cdot))$ of the minimization problem in (\ref{controlcost}) satisfies
\begin{equation}\label{Pontry}
\dot x_i=\frac{\partial H_u^\nu}{\partial p_i},\quad \dot p_i=-\frac{\partial H_u^\nu}{\partial x_i},\quad \frac{\partial H_u^\nu}{\partial u_i}=0
\end{equation}
for some curve $p(\cdot)$ and some constant $\nu$ such that $(\nu,p(t))\neq 0$. Moreover, $\nu$ can be chosen to be either $0$ or $-1$. A minimizer $(x(\cdot), u(\cdot))$ is abnormal if the corresponding $\nu$ in the Pontryagin maximum principle is $0$. It is normal if $\nu=-1$. Note that a minimizer can both be normal and abnormal. 

In the present case, the Hamiltonian $H_u^\nu$ is given by
\[
H_u^\nu(x,p)=p_1u_1+p_2x_1^2+p_2u_2x_1^k+\frac{\nu}{2}(u_1^2+u_2^2).
\]

For the abnormal case $\nu=0$, (\ref{Pontry}) becomes 
\[
\dot x_1=u_1,\quad \dot x_2=x_1^2+u_2x_1^k, \quad \dot p_2=0, \quad p_1=0, \quad p_2x_1^k=0. 
\]
By Pontryagin maximum principle, $p_1$ and $p_2$ cannot be equal to zero simultaneously since $\nu=0$. It follows that $x_1\equiv 0$ and $x_2\equiv x_2(0)$. The corresponding controls to all these paths are all given by the zero control $u\equiv 0$. It follows that $c_1((0,w),(0,w))=0$ for any $w$ and these are candidates for discontinuities of the cost $c_1$.

Next, we show that $c_1((0,w),(0,z))\geq K$ for some constant $K>0$ and for all $z<w$. For this, we consider the case $\nu=-1$. In this case, the Hamiltonian is given by
\[
H_u^{-1}(x,p)=p_1u_1+p_2x_1^2+p_2u_2x_1^k-\frac{1}{2}(u_1^2+u_2^2).
\]
It follows from (\ref{Pontry}) that we have
\begin{equation}\label{usePontry}
\begin{split}
&H:=H^{-1}_u(x,p)=\frac{1}{2}p_1^2 +\frac{1}{2}x_1^{2k}p_2^2+x_1^2p_2,\quad
\\&\dot x_1=p_1, \quad \dot p_2=0, \quad u_1=p_1,\quad u_2=x_1^kp_2.
\end{split}
\end{equation}

If we assume that $x(0)=(0,w)$ and $x(1)=(0,z)$ with $z<w$, then it follows from (\ref{usePontry}) that the cost $c_1((0,w),(0,z))$ for going from $(0,w)$ to $(0,z)$ is estimated by
\begin{equation}\label{costest}
\begin{split}
c_1((0,w),(0,z)) &=\frac{1}{2}\int_0^1p_1^2+x_1^{2k}p_2^2dt\geq \frac{1}{2}\int_0^1p_1^2dt.
\end{split}
\end{equation}

Since $p_2$ is a constant of motion, we can fix $p_2$ and look at the phase portrait of the system
\[
\dot x_1=p_1,\quad\dot p_1=-kx_1^{2k-1}p_2^2-2x_1p_2
\]
(see Figure 1). The cost $c((0,w),(0,z))$ in (\ref{costest}) can be estimated from below by the area enclosed by the level set $H=0$. More precisely, 
\begin{equation}\label{costest2}
c((0,w),(0,z)) \geq \int_0^{\kappa}p_1(x_1,p_2)dx_1,
\end{equation}
where $p_1(x,p_2)$ is defined implicitly by $\frac{1}{2}p_1^2 +\frac{1}{2}x_1^{2k}p_2^2+x_1^2p_2=0$ and $\kappa=\left(-\frac{2}{p_2}\right)^{\frac{1}{2k-2}}$ is the positive zero of the function $p_1(x_1,p_2)$. Note that $p_2<0$. Indeed, $H(x(t),p(t))$ is constant, we have $H(x(t),p(t))=H^{-1}_u(x(0),p(0))\geq 0$. It follows that $p_2(t)=p_2(0)\leq 0$. 

\begin{figure}[ht!]\label{phase}
\input epsf
\centerline{\epsfysize=0.5\vsize\epsffile{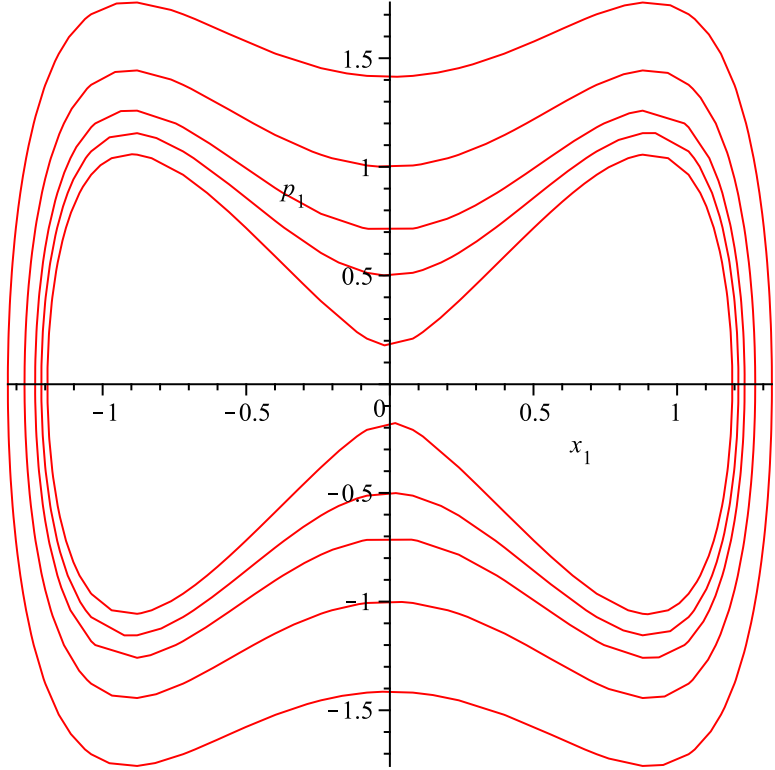}}
\caption{Phase portrait of the example (level sets of the Hamiltonian $H$)}
\end{figure}

If we do a change of variable $x_1=\kappa z$, then we have
\begin{equation}\label{costest3}
\begin{split}
\int_0^\kappa p_1(x_1,p_2)dx_1 & \geq \int_0^\kappa\left(-2 x_1^2p_2 -x_1^{2k}p_2^2\right)^{\frac{1}{2}}dx_1
\\& =2^{\frac{k+1}{2k-2}}\left(-p_2\right)^{\frac{k-3}{2k-2}}\int_0^1(z^2-z^{2k})^{\frac{1}{2}}dz.
\end{split}
\end{equation}

On the other hand, by Figure 1 and (\ref{usePontry}), we have 
\begin{equation}\label{costest4}
\frac{1}{2}\int_0^1p_1^2dt\geq \frac{1}{2}\int_0^1|p_1|dt\geq \frac{1}{2}\int_{\{t|\dot x_1(t)\geq 0\}}\dot x_1dt\geq \frac{1}{2}\left(-\frac{2}{p_2}\right)^{\frac{1}{2k-2}}.
\end{equation}

If we combine  (\ref{costest}), (\ref{costest2}), (\ref{costest3}), and (\ref{costest4}), then we get
\[
c((0,w),(0,z))\geq C \max\left\{\left(-\frac{1}{p_2}\right)^{\frac{1}{2k-2}},\left(-p_2\right)^{\frac{k-3}{2k-2}}\right\}
\] 
for some constant $C>0$. 

It follows that the cost $c((0,w),(0,z))$ is bounded below by a positive constant independent of $p_2$ if $k\geq 3$ and this finishes the proof of the result. 
\end{proof}

\

\section{Continuity of Optimal Control Costs}

In this section, we will state and prove the general continuity result (Theorem \ref{morecontinuity}) mentioned in the introduction. To do this, let us introduce some notations. If $X_t$ is a, possibly time-dependent, vector field, then the corresponding flow $\varphi_t$ defined by $\varphi_0(x)=x$ and $\ddt\varphi_t(x)=X_t(\varphi_t(x))$ is denoted by
\[
\varphi_t=\flow 0 tX_sds.
\]

We define the endpoint map $End_{x_0}^T:L^p([0,T],U)\to M$ by
\[
 End_{x_0}^T(u(\cdot))=\flow 0 T F_{u(s)}ds(x_0),
\]
where $F_u$ is the vector field defined by $F_u(x)=F(x,u)=X_0(x)+\sum_{i=1}^nu_iX_i(x)$.

Let us first fixed a control $u(\cdot)$. The first goal is to show that the control system is locally controllable. It means that we can reach any point near the point $End_{x_0}^T(u(\cdot))$ by adding a small control $v(\cdot)$ to the fixed one $u(\cdot)$. The first idea is to replace the control system (\ref{controlaff}) with drift $X_0$ by one without drift. However, the control vector fields $X_1,...,X_n$ will become time dependent in the new control system. This is accomplished in Lemma \ref{time-dep-linear}. Recall that if $P:M\to M$ is a diffeomorphism and $X$ is a vector field on $M$, then the pull back vector field $P^*X$ is the vector field defined by $P^*X=dP^{-1}(X\circ P)$. 

\begin{lem}\label{time-dep-linear}
Let $g_i^t$ be the time-dependent vector field defined by
\[
g_i^t:=\left(\flow 0 t F_{u(s)} ds\right)^*X_i.
\]
Then
\[
 End^T_{x_0}(u(\cdot)+v(\cdot))=\flow 0 T F_{u(t)}dt\circ \flow 0 T \sum_{i=1}^n v_i(t)g_i^tdt(x_0).
\]
\end{lem}

\begin{proof}
Let $Q_t$ and $R_t$ be the flows $\flow 0 t F_{u(s)+v(s)}ds$ and $\flow 0 t F_{u(s)}ds$, respectively. Let $P_t$ be the flow defined by $Q_t=R_t\circ P_t$. If we differentiate the above equation, then we get
\[
F_{u(t)+v(t)}\circ Q_t=\dot Q_t=\dot R_t\circ P_t+dR_t(\dot P_t)=F_{u(t)}\circ Q_t+dR_t(\dot P_t).
\]
After simplifying the above equation, we get $\dot P_t=dR_t^{-1}(F_{v(t)}\circ Q_t)=(R_t)^*F_{v(t)}\circ P_t$ and this completes the proof.
\end{proof}

Recall that we want to show local controllability by varying $v(\cdot)$. In Lemma \ref{time-dep-linear}, we have decompose the endpoint map $End_{x_0}(u(\cdot)+v(\cdot))$ into two parts.  The first part $\flow 0 T F_{u(t)}dt$ is independent of the varying control $v(\cdot)$ and it is a diffeomorphism. Therefore, it is enough to show local controllability for the second term $\flow 0 T \sum_{i=1}^n v_i(t)g_i^tdt(x_0)$ which is the endpoint map to a new control system 
\begin{equation}\label{newcontrol}
\dot x=\sum_{i=1}^n v_i(t)g_i^t.
\end{equation}
Note that this is a system with no drift but with time dependent control vector fields $g_i^t$ as mentioned earlier. 

Before proceeding to the proof of local controllability of the system (\ref{newcontrol}), let us state the main result of this section which includes Theorem \ref{continuity} as a corollary.

\begin{thm}\label{morecontinuity}
Assume that the Lagrangian $L$ and the family of vector fields
\[
\{g_1^t,...,g_n^t|t\in [0,T]\}
\]
satisfy the following conditions:
\begin{enumerate}
\item $C_1|u|^q+K_1\leq L(x,u)\leq C_2|u|^p+K_2$,
\item $|\frac{\partial L(x,u)}{\partial x}|\leq C_3|u|^2$,
\item  the Hessian of $L$ in the $u$ variable is positive definite, and
\item $\{g_1^t,...,g_n^t \mid t\in [0,T]\}$ is $k$-generating, $\forall u(\cdot)$,
\end{enumerate}
for some constants $C_1,C_2,C_3,K_1,K_2>0$ and some constant $q>1$. Suppose further that one of the followings is satisfied:
\begin{enumerate}
 \item $k=3$ and $p\leq 2$, or
\item $k>3$, $p<\frac{k-2}{k-3}$.
\end{enumerate}
Then the cost function $(t,x,y)\mapsto c_t(x,y)$ defined in (\ref{controlcost}) is continuous.
\end{thm}

Going back to the local controllability issue of the system (\ref{newcontrol}), let us denote the endpoint map to the new system by  $\Phi^T:L^p([0,T],U)\times M\to M$. More precisely, 
\[
 \Phi^T(v(\cdot),x):=\flow 0 T \sum_{i=1}^n v_i(t)g_i^tdt(x).
\]

If the control vector fields $g_i^t$ in the above new system is time independent, then local controllability follows from the Chow-Rashevskii theorem (see for instance \cite{Mo}). More precisely, we will need the following lemma for which the proof will be given for completeness. Recall that if $X,Y$ are two vector fields, then the vector field $ad_XY$ is defined by $ad_XY=[X,Y]$. 

\begin{lem}\label{Chow}
Let $g_1,...,g_n$ which are time-independent family of vector fields. Then there exists piecewise constant control $w(\cdot)$ for which $w(t)$ has only one nonzero component for each $t$ and such that
\begin{equation}\label{Chow-1}
\begin{split}
 & f\left(\flow 0 T \e\left(\sum_{i=1}^nw_i(t)g_i\right) dt(x_0))\right)\\&=f(x_0)+\e^k(ad_{g_1}...ad_{g_{k-1}}g_{k}) f(x_0)+o(\e^k)
 \end{split}
\end{equation}
as $\e\to 0$ for every smooth function $f$. 
\end{lem}

\begin{proof}
Let $P_t^\e$ and $Q_t^\e$ be the flows corresponding to the control system (\ref{controlaff}) with controls $\e w^P$ and $\e w^Q$, respectively. More precisely, 
\[
P^\e_t(x_0)=\flow 0 t\e\left(\sum_{i=1}^nw_i^P(s)g_i\right) ds, \quad Q_t^\e(x_0) =\flow 0 t\e\left(\sum_{i=1}^nw_i^Q(s)g_i\right) ds. 
\]

Moreover, assume that there are vector fields $X$ and $Y$ such that the flows $P_t^\e$ and $Q_t^\e$ satisfy 
\[
f(P_t^\e(x_0))=f(x_0)+\e Xf(x_0)+o(\e), \quad f(Q_t^\e(x_0))=f(x_0)+\e^k Yf(x_0)+o(\e^k)
\]
for all smooth functions $f$. 

Next, we define a control $\bar w$ which is the concatenation of the controls $w^P$, $w^Q$, $-w^P$, and $-w^Q$. 
\[
\bar w(t)=
\begin{cases}
-w^P\left(t\right) & \text{ if } 0\leq t\leq T\\
-w^Q\left(t-T\right) & \text{ if } T<t\leq 2T\\
w^P(t-2T) & \text{ if } 2T< t\leq 3T\\
w^Q(t-3T) & \text{ if } 3T< t\leq 4T.\\
\end{cases}
\]

It follows that 
\[
f\left(\flow 0 {4T}\e\left(\sum_{i=1}^n\bar w_i(s)g_i\right) ds(x_0)\right)=f(Q_T^\e\circ P_T^\e\circ (Q_T^\e)^{-1}\circ (P_T^\e)^{-1}(x_0)). 
\]

Let $h(\e_1,\e_2)=f(Q_T^{\e_2}\circ P_T^{\e_1}\circ(Q_T^{\e_2})^{-1}\circ (P_T^{\e_1})^{-1}(x_0))$ and we want to consider the expansion of the function $h(\e,\e)$ in the parameter $\e$. Note that $P_T^0=Q_T^0$ is the identity transformation. It follows that the zeroth order term of the expansion of $h(\e,\e)$ in $\e$ is $f(x_0)$. In fact, the following is true. 
\begin{equation}\label{Chow-2}
h(\e_1,0)=h(0,\e_2)=f(x_0). 
\end{equation}

By definition of the flow $Q_t^\e$, we have $\partial_\e^if(Q_t^\e)\Big|_{\e=0}=0$ for all $i=1,...,k-1$. It follows that $\partial_{\e_2}^ih\Big|_{\e_2=0}=0$ for each such $i$. Therefore, except the zeroth order term, any term of order less than $k$ in the expansion of $h$ vanishes. However, by (\ref{Chow-2}), the $k$-th order vanishes as well. Therefore, we consider the $(k+1)$-th order term. Moreover, by the same argument, the only nontrivial $(k+1)$-th order term is given by $\partial_{\e_1}\partial_{\e_2}^kh\Big|_{\e_1=\e_2=0}$. A computation shows the following 
\[
\partial_\e^{k+1} h(\e,\e)\Big|_{\e=0} =(k+1)\partial_{\e_1}\partial_{\e_2}^{k} h(\e_1,\e_2)\Big|_{\e_1 =\e_2=0}=(k+1)[Y,X]f(x_0). 
\]

In conclusion, we have shown that 
\[
f\left(\flow 0 {4T}\e\left(\sum_{i=1}^n\bar w_i(s)g_i\right) ds(x_0)\right) =f(x_0)+\frac{\e^{k+1}}{k!}[Y,X]f(x_0)+o(\e^{k+1}). 
\]

By rescaling time and multiplying the control $\bar w$ by a constant, we have a control $w$ which satisfies 
\begin{equation}\label{Chow-3}
f\left(\flow 0 {T}\e\left(\sum_{i=1}^n w_i(s)g_i\right) ds(x_0)\right) =f(x_0)+\e^{k+1}[Y,X]f(x_0)+o(\e^{k+1}). 
\end{equation}
Note that if the controls $w^P$ and $w^Q$ are piecewise constant and have only one nonzero component for each time $t$, then so is $w$ by construction. 

If we let the control $w^P$ and $w^Q$ be the constant controls defined by $w^P_{i}(t)=\delta_{i,i_1}$ and $w^Q_{i}(t)=\delta_{i,i_2}$ for each $t$, then (\ref{Chow-3}) shows that 
\[
f\left(\flow 0 {T}\e\left(\sum_{i=1}^n w_i(s)g_i\right) ds(x_0)\right) =f(x_0)+\e^2[g_{i_2},g_{i_1}]f(x_0)+o(\e^{2}). 
\]
This proves the lemma for the case $k=2$. The rest follows from induction using (\ref{Chow-3}). 
\end{proof}

The second idea is to take a control given by Lemma \ref{Chow}, rescale it so that it is concentrated on a smaller and smaller time interval, and put the rescaled controls to the place where the vector fields $g_1^t,...,g_n^t$ are bracket generating. This way we obtain local controllability as in Chow-Rashevskii theorem. Here we need the conditions on the numbers $k$ and $p$ to make sure that the rescaled controls stay small. This second idea will be achieved in Proposition \ref{infin} below. To do this, let us consider the curves $t\mapsto g_i^t(x_0)$ contained in the tangent space $T_{x_0}M$. Let $\mathcal I$ be an interval in $[0,T]$ with the property that any subinterval $\mathcal I'$ contained in $\mathcal I$ satisfies
\[
\text{span}\{g_1^t(x_0),...,g_n^t(x_0)|t\in \mathcal I'\}=\text{span}\{g_1^t(x_0),...,g_n^t(x_0)|t\in\mathcal I\}.
\]

\begin{prop}\label{infin}
Let $\tau$ be a Lebesgue point of the control $u(\cdot)$ contained in the interval $\mathcal I$ and assume that either
\begin{enumerate}
 \item $k=3$ and $p\leq 2$, or
\item $k>3$, $p<\frac{k-2}{k-3}$.
\end{enumerate}
Then  there exists $\alpha,\beta>0$ and a family of controls $v^\e(\cdot)$ which converges to $0$ in $L^p$ and such that
\[
f(\Phi^T(v^\e(\cdot)),x_0)=f(x_0)+\e^{k(\beta-\alpha)}\int_0^T ad_{g_{i_1}^\tau}...ad_{g_{i_{k-1}}^\tau}(g_{i_k}^\tau)f(x_0)ds+o(\e^{k(\beta-\alpha)}),
\]
as $\e\to 0$, for any smooth function $f$.
\end{prop}

\begin{proof}
By Lemma \ref{Chow}, there is a piecewise constant control $w(\cdot)$  for which $w(t)$ has only one nonzero component for each $t$ and such that
\begin{equation}\label{rescale}
\begin{split}
 & f\left(\flow 0 T \e\left(\sum_{i=1}^nw_i(t)g_i^\tau\right) dt(x_0))\right)\\&=f(x_0)+\e^k(ad_{g_{i_1}^\tau}...ad_{g_{i_{k-1}}^\tau}g_{i_k}^\tau) f(x_0)+o(\e^k)
 \end{split}
\end{equation}
as $\e\to 0$. Note that $\tau$ is fixed and $g_i^\tau$ is a time independent vector field. 

Let $0=t_0\leq t_1\leq...\leq t_l=T$ be a partition such that the restriction $w|_{[t_{i-1},t_i)}$ of the control $w(\cdot)$ to the subinterval $[t_{i-1},t_i)$ is constant and there is only one nonzero component. We suppose that the $k_i$-th component of $w|_{[t_{i-1},t_i)}$ is nonzero and this nonzero component is equal to $c_i$.

We need to create more freedom in our controls for later use (Lemma \ref{I1} to be precise). Let $v(\cdot)$ be a control of the form $v(\cdot)=w(\cdot)+\alpha(\cdot)$ such that
$\alpha_j|_{[t_i,t_{i+1})}\equiv 0$ if $j\neq k_i$ and $\int_{t_i}^{t_{i+1}}\alpha_{k_i}(s)ds=0$. It follows from (\ref{rescale}) and $\int_{t_i}^{t_{i+1}}\alpha_{k_i}(s)ds=0$ that

\begin{equation}\label{rescaleagain}
\begin{split}
 &f\left(\flow 0 T \e\left(\sum_{i=1}^nv_i(t)g_i^\tau\right) dt(x_0))\right)\\&=f(x_0)+\e^k(ad_{g_{i_1}^\tau}...ad_{g_{i_{k-1}}^\tau}g_{i_k}^\tau) f(x_0)+o(\e^k).
\end{split}
\end{equation}

Next, we rescale the control $v(\cdot)$ as mentioned earlier. Let $G_{s,v}:=\sum_{i=1}^n v_ig_i^s$ and let
\[
v^\e(t)=
\begin{cases}
\e^{-\alpha}v((t-\tau)/\e^\beta) & \text{ if } t\in (\tau,\tau+\e^\beta)\\
0 &\text{ otherwise}.
\end{cases}
\]

Then we have
\[
\begin{split}
 & f\left(\flow 0 T G_{s,v^\e(s)}ds(x_0)\right)\\&= f\left(\flow \tau {\tau+\e^\beta} G_{s,v^\e(s)}ds(x_0)\right)
 \\&= f\left(\flow 0 T \e^{\beta-\alpha}\sum_{i=1}^nv_i(s)g_i^{\e^\beta s+\tau}ds(x_0)\right)
 \\&= f\left(\flow 0 T \e^{\beta-\alpha}G_{\e^\beta s+\tau,v(s)}ds(x_0)\right).
\end{split}
\]

By using the asymptotic expansion in \cite[section 2.4.4]{AgSa}, the above equation becomes
\begin{equation}\label{expand}
\begin{split}
 & f\left(\flow 0 T G_{s,v^\e(s)}ds(x_0)\right)
 \\&=f(x_0)+\sum_{i=1}^{k}\int_{0\leq s_1\leq ...\leq s_i\leq T} \e^{i(\beta-\alpha)}G_{\e^\beta s_1+\tau,v(s_1)}...G_{\e^\beta s_i+\tau,v(s_i)}f(x_0)ds_1...ds_i+\\&+o(\e^{k(\beta-\alpha)}).
\end{split}
\end{equation}
as $\e\to 0$.

Let $I_i$ be the term
\[
 I_i(v(\cdot)):=\int_{0\leq s_1\leq ...\leq s_i\leq T} G_{\e^\beta s_1+\tau,v(s_1)}...G_{\e^\beta s_i+\tau,v(s_i)}f(x_0)ds_1...ds_i
\]
in the expansion (\ref{expand}).

Let us first deal with the term $I_1(v(\cdot))=\int_0^T G_{\e^\beta s+\tau,v(s)}f(x_0)ds$. For this, let $g_i^{\tau+\e s}=g_i^\tau+Z_i^{\e,s}$. Let us recall that $v(\cdot)=w(\cdot)+\alpha(\cdot)$ and $\alpha_j\equiv 0$ if $j\neq k_i$. 

\begin{equation}\label{preI1}
\begin{split}
I_1(v(\cdot))&=\int_0^T\sum_{i=1}^n v_i(s)g_i^{\tau+\e^\beta s}f(x_0)ds
\\ &=\int_0^T\sum_{i=1}^n v_i(s)g_i^{\tau}f(x_0)ds+\int_0^T\sum_{i=1}^n (w_i(s)+\alpha_i(s))Z^{\e,s}_if(x_0)ds
\\ &=\int_0^T\sum_{i=1}^n v_i(s)g_i^{\tau}f(x_0)ds +\sum_{i=1}^n \int_{t_{i-1}}^{t_i} (c_i+\alpha_{k_i}(s))Z^{\e, s}_{k_i}f(x_0)ds.
\end{split}
\end{equation}

The next lemma says that we can choose $\alpha$ to get rid of the last term of the above equation. 

\begin{lem}\label{I1}
There exists $\e_0>0$ and $\alpha^\e(\cdot)$ in $L^2$ such that $\int_{t_{i-1}}^{t_i}\alpha^\e_{k_i}(s)ds=0$ and
\begin{equation}
I_1(w(\cdot)+\alpha^\e(\cdot))=\int_0^T\sum_{i=1}^n v_i(s)g_i^{\tau}f(x_0)ds
\end{equation}
for all $0<\e<\e_0$.
Moreover, $\alpha^\e$ goes to 0 in $L^2$ as $\e$ goes to 0.
\end{lem}

\begin{proof}[Proof of Lemma \ref{I1}]
Recall that we need $\alpha^\e(\cdot)$ to satisfy the conditions
\begin{equation}\label{conditions}
\int_{t_{i-1}}^{t_i}\alpha^\e_{k_i}(s)ds=0, \quad \int_{t_{i-1}}^{t_i} (c_i+\alpha^\e_{k_i}(s))Z^{\e, s}_{k_i}f(x_0)ds=0.
\end{equation}
for all smooth functions $f$ and for all $i$. Consider local coordinates around the point $x_0$ and suppose that $Z^{\e, s}_{k_i}=(Z^{\e, s}_{k_i,1},...,Z^{\e, s}_{k_i,m})$  in this local coordinates. Then the conditions in (\ref{conditions}) is the same as that $\alpha^\e_{k_i}(\cdot)$ orthogonal to the constant functions and $c_i+\alpha^\e_{k_i}(\cdot)$ is orthogonal to $Z^{\e, \cdot}_{k_i}$ in $L^2([t_{i-1},t_i])$ for each $i$. Let $V_i$ be the finite dimensional subspace of $L^2([t_{i-1},t_i])$ defined by
\[
 V_i^\e:=\text{span}\{Z^{\e, \cdot}_{k_i,1},...,Z^{\e, \cdot}_{k_i,m}\}.
\]

A linear algebra argument shows that $\alpha^\e(\cdot)$ which satisfy the conditions (\ref{conditions}) exist if $V_i^\e$ does not contain any nonzero constant function. Therefore, we assume that
\[
\sum_{j=1}^ma_jZ^{\e, s}_{k_i,j}=c
\]
for some constants $a_1,...,a_m,c$, for all $s$ in $[t_{i-1},t_i]$, and for some $i$. We are going to show that $c$ must be zero and this finishes the proof of the lemma.

The above equation means that $(Z^{\e, s}_{k_i,1},...,Z^{\e, s}_{k_i,m})$ is contained in the affine space $\{z\in\Real^m|\sum_{i=1}^ma_iz_i=c\}$ for almost all $s$ in the interval $[t_{i-1},t_i]$. Therefore, $(\dd sZ^{\e, s}_{k_i,1},...,\dd sZ^{\e, s}_{k_i,m})$ is contained in the subspace $\{z\in\Real^m|\sum_{i=1}^ma_iz_i=0\}$ for each $s$ in the interval $[t_{i-1},t_i]$. Let us choose $\e_0$ such that $\tau+t\e$ is contained in the interval $\mathcal I$ for each $t$ in $[0,T]$ and for all $\e<\e_0$. Then it follows from the definition of the interval $\mathcal I$ that $(\dd sZ^{\e, s}_1(x_0),...,\dd sZ^{\e, s}_m(x_0))$ is contained in $\{z\in\Real^m|\sum_{i=1}^ma_iz_i=0\}$ for almost all $s$ in $[0,T]$. Therefore, $(Z^{\e, s}_{k_i,1}(x_0),...,Z^{\e, s}_{k_i,m}(x_0))$ is contained in the affine space $\{z\in\Real^m|\sum_{i=1}^ma_iz_i=c\}$ for all $s$ in $[0,T]$. However, $(Z^{\e, 0}_{k_i,1},...,Z^{\e, 0}_{k_i,m})=0$, so $c=0$ and this finishes the proof of the lemma.
\end{proof}

For the rest of the proof, we write $v(\cdot)=w(\cdot)+\alpha^\e(\cdot)$ and suppress the $\e$-dependence on $v$ to avoid complicated notation.

\begin{lem}\label{Ik}
\[
I_k(v(\cdot))=\int_{0\leq s_1\leq ...\leq s_k\leq T}\sum_{i_1,...,i_k=1}^n v_{i_1}(s_1)...v_{i_k}(s_{i_k})g_{i_1}^{\tau}...g_{i_k}^{\tau}f(x_0)ds_1...ds_k+o(1)
\]
as $\e\to 0$.
\end{lem}

\begin{proof}[Proof of Lemma \ref{Ik}]
This follows immediately from the definition of $G_{\tau,v}$. Indeed,
\begin{equation}\label{I2}
\begin{split}
&I_k(v(\cdot))\\&:=\int_{0\leq s_1\leq ...\leq s_k\leq T} G_{\e^\beta s_1+\tau,v(s_1)}...G_{\e^\beta s_k+\tau,v(s_k)}f(x_0)ds_1...ds_k\\
&=\int_{0\leq s_1\leq ...\leq s_k\leq T}\sum_{i_1,...,i_k=1}^n v_{i_1}(s_1)...v_{i_k}(s_k)g_{i_1}^{\tau+\e^\beta s_1}...g_{i_k}^{\tau+\e^\beta s_k}f(x_0)ds_1...ds_k\\
&=\int_{0\leq s_1\leq ...\leq s_k\leq T}\sum_{i_1,...,i_k=1}^n v_{i_1}(s_1)...v_{i_k}(s_{i_k})g_{i_1}^{\tau}...g_{i_k}^{\tau}f(x_0)ds_1...ds_k+o(1).
\end{split}
\end{equation}
\end{proof}

If we combine Lemma \ref{I1} and Lemma \ref{Ik} with (\ref{expand}) and assume that $3\beta-2\alpha>k(\beta-\alpha)>0$, then we have
\[
\begin{split}
 & f\left(\flow 0 T G_{s,v^\e(s)}ds(x_0)\right)
 \\&=f(x_0)+\sum_{i=1}^k\e^{i(\beta-\alpha)}I_i(w(\cdot)+\alpha^\e(\cdot))+O(\e^{k(\beta-\alpha)})
  \\&=f(x_0)+\sum_{j=1}^k\e^{j(\beta-\alpha)}\int_{0\leq s_1\leq ...\leq s_j\leq T}\sum_{i_1,...,i_j=1}^n v_{i_1}(s_1)...v_{i_j}(s_{i_j})g_{i_1}^{\tau}...g_{i_j}^{\tau}f(x_0)ds_1...ds_j
  \\&+o(\e^{k(\beta-2\alpha)})
\end{split}
\]
as $\e\to 0$.

By (\ref{rescaleagain}), the above becomes
\[
\begin{split}
 & f\left(\flow 0 T G_{s,v^\e(s)}ds(x_0)\right)
  \\&=f\left(\flow 0 T \e^{\beta-\alpha}\sum_{i=1}^nv_i(t)g_i^\tau dt\right)+o(\e^{k(\beta-\alpha)})
   \\&=f(x_0)+\e^{k(\beta-\alpha)}(ad_{g_{i_1}}^\tau ...ad_{g_{i_{k-1}}}g_{i_k})f(x_0)+o(\e^{k(\beta-\alpha)}).
\end{split}
\]

Finally, we need $v^\e(\cdot)$ converges to 0 in $L^p$. Indeed, by the definition of $v^\e(\cdot)$, we have
\[
\begin{split}
\int_0^T|v^\e(t)|^pdt&=\int_\tau^{\tau+\e^\beta}\left|\e^{-\alpha}v\left(\frac{t-\tau}{\e^\beta}\right)\right|^pdt
\\&=\int_0^1\left|v\left(s\right)\right|^p\e^{\beta-\alpha p} ds
\\&=\int_0^1\left|w\left(s\right)+\alpha^\e\left(s\right)\right|^p\e^{\beta-\alpha p} ds.
\end{split}
\]
Since $w(\cdot)$ is in $L^\infty$ and $\alpha^\e(\cdot)$ is in $L^2$, $v^\e(\cdot)$ converges to 0 in $L^p$ if $\beta-\alpha p>0$ and $p\leq 2$.

In conclusion, if we can choose $\alpha$ and $\beta$ such that the following three conditions are satisfied, then the conclusion of the theorem holds.
\[
 3\beta-2\alpha>k(\beta-\alpha)>0,\quad \beta-\alpha p>0,\quad p\leq 2.
\]

It is not hard to check that these inequalities are satisfied under the assumptions of the proposition.
\end{proof}

The local controllability of the control system follows using Proposition \ref{infin} and implicit function theorem as in the Chow-Rashevskii theorem. Finally, the continuity of the cost follows from the local controllability  and  standard arguments as in \cite{CaRi}. 

\begin{proof}[Proof of Theorem \ref{morecontinuity}]
Lower semi-continuity of the cost can be proved in the same way as in \cite{CaRi}. To prove upper semi-continuity, we let $(x_1,y_1,t_1),(x_2,y_2,t_2),...$ be a sequence of points which converges to $(x,y,T)$ and $\lim_{i\to\infty}c_{t_i}(x_i,y_i)=r$. We want to show that $c_T(x,y)\geq r$. 

Assume that this is not the case. Let $u(\cdot)$ and $x(\cdot)$  be a control and the trajectory associated to this control, respectively, such that $x(0)=x,\ x(T)=y$ and $\int_0^TL(x(s),u(s))ds<r$. Recall that the family of vector fields $\{g_1^t,...,g_n^t|t\in[0,T]\}$ is $k$-generating. Therefore, we can find vector fields $V_1,...,V_k$ from the the set
\[
\left\{ad_{g_{i_1}^t}...ad_{g_{i_l}^t}\Big|1\leq i_j\leq n,0\leq t\leq T,0\leq l\leq k\right\}.
\]
which span the tangent space $T_x M$. We also assume that $V_i$ is defined by the Lie brackets of $\kappa_i$ vector fields of the form $g_j^{\tau_i}$. By perturbation, we can assume $\tau_i\neq \tau_j$ for $i\neq j$ and that each $\tau_i$ satisfies the condition in Theorem \ref{infin}. Therefore, by Theorem \ref{infin}, there is a family of control $w_{i,\e}(\cdot)$ such that
\[
f(\Phi^T(w_{i,\e}(\cdot),x_0))=f(x_0)+\e^{\kappa_i(\beta-\alpha)}\int_0^T V_if(x_0)ds+o(\e^{\kappa_i(\beta-\alpha)}).
\]

Note that from the proof of Proposition \ref{infin}, we can assume that $w_{i,\e_i}$ is supported in a small interval $J_i$ around $\tau_i$ by taking $(\e_1,...,\e_n)$ small enough. Moreover, we can assume that the intervals $J_i$ are disjoint. We define the map $\Psi:M\times\Real^n\times (0,\infty)\to M$ by
\[
\Psi(x,\e_1,...,\e_n,T)=\Phi^T(w_{1,\e_1/\kappa_1(\beta-\alpha)})\circ\Phi^T(w_{2,\e_2/\kappa_2(\beta-\alpha)})\circ...\circ\Phi^T(w_{n,\e_n/\kappa_n(\beta-\alpha)})(x).
\]

Since $\dd{\e_i}\Big|_{\e_i=0}\Psi=V_i$, the map $\Psi$ is of full rank at the point $(x,0,...,0,T)$. It follows from implicit function theorem that there exists a map $\psi:U_1\to U_2$ from a neighborhood $U_1$ of $(x,y,T)$ to a neighborhood $U_2$ of $(0,...,0)$ such that $\Psi(z_1,\psi(z_1,z_2,t),t)=z_2$ for all pairs $(z_1,z_2,t)$ in the set $U_1$.

Let $(\e_1^i,...,\e_n^i)=\psi(x_i,y_i,t_i)$ and let $v^i(\cdot)$ be the control defined by $v^i(t)=w_{j,\e_j^i}(t)$ and $0$ otherwise. $v^i(\cdot)$ is well defined if $(\e_1^i,...,\e_n^i)$ is close enough to $(0,...,0)$. Let $x^i(\cdot)$ be a curve in $M$ which satisfies (\ref{controlaff}) with control $v^i(\cdot)$. We know that $v^i(\cdot)$ converges strongly in $L^p$ to $0$ and $x^i(\cdot)$ converges uniformly to $x(\cdot)$.

Assume, without loss of generality, that $u(t)=0$ for all $t>T$. Then
\begin{equation}\label{C1}
\begin{split}
&\left|\int_0^{t_i}L(x^i(s),u(s)+v^i(s))ds-\int_0^{T}L(x(s),u(s))ds\right|\\&  \leq \left|\int_0^{t_i}L(x^i(s),u(s)+v^i(s))ds-\int_0^{t_i}L(x(s),u(s)+v^i(s))ds\right|+\\& +\left|\int_0^{\tau_i}L(x(s),u(s)+v^i(s))ds-\int_0^{\tau_i}L(x(s),u(s))ds\right|+\\& +\left|\int_{t_i}^T\max\{L(x(s),u(s)+v_i(s)),L(x(s),u(s))\}ds\right|,
\end{split}
\end{equation}
where $\tau_i=\min \{T,t_i\}$.

Since $L(x,u)\leq C_2|u|^p+K_2$, $t_i$ converges to $T$, and the sequence  $u(\cdot)+v^i(\cdot)$ converges to $u(\cdot)$ in $L^p$, we have
\begin{equation}\label{C2}
\begin{split}
&\left|\int_{t_i}^T\max\{L(x(s),u(s)+v^i(s)),L(x(s),u(s))\}ds\right|
\\&\leq \int_{t_i}^T\max\{C_2|u(s)+v^i(s)|^p-K_2, C_2|u(s)|^p-K_2\}ds
\\& \to 0
\end{split}
\end{equation}
as $i\to\infty$.

Recall that $|\frac{\partial L}{\partial x}|\leq C_3|u|^2$, where the norm is taken with respect to certain Riemannian metric. Let $d$ be the corresponding Riemannian distance function. Then we have
\begin{equation}\label{C3}
\begin{split}
&\left|\int_0^{t_i}L(x^i(s),u(s)+v^i(s))ds-\int_0^{t_i}L(x(s),u(s)+v^i(s))ds\right|
\\& \leq \int_0^{t_i}\left|L(x^i(s),u(s)+v^i(s))ds-L(x(s),u(s)+v^i(s))\right|ds
\\&
\leq \sup_sd(x^i(s),x(s))\,\int_0^{t_i}C_3\left|u(s)+v^i(s)\right|^2ds\\&\to 0 \quad \quad \text{ as } i\to\infty.
\end{split}
\end{equation}

By construction of the control $v^i(\cdot)$, we know that the indicator function $\mathbb I_{\{t|v_i(t)\neq 0\}}$ converges to zero almost everywhere. It follows that
\begin{equation}\label{C4}
\begin{split}
&\left|\int_0^{\tau_i}L(x(s),u(s)+v^i(s))ds-\int_0^{\tau_i}L(x(s),u(s))ds\right|\\&
\leq \int_0^{\tau_i}\mathbb I_{\{t|v^i(t)\neq 0\}}(\left|L(x(s),u(s)+v^i(s))\right|+\left|L(x(s),u(s))\right|)ds\\& \to\ 0\quad \quad \text{ as } i\to\infty.
\end{split}
\end{equation}

Therefore, if we combine (\ref{C1}), (\ref{C2}), (\ref{C3}), and (\ref{C4}), then we have
\[
\lim_{i\to\infty}\int_0^{t_i}L(x^i(s),u(s)+v^i(s))ds=\int_0^TL(x(s),u(s))ds<r.
\]
On the other hand,
\[
\lim_{i\to\infty}\int_0^{t_i}L(x^i(s),u(s)+v^i(s))ds\geq \lim_{i\to\infty}c_{t_i}(x_i,y_i)=r.
\]

Therefore, this gives a contradiction and we finish the proof of upper semi-continuity of the function $(t,x,y)\mapsto c_t(x,y)$.
\end{proof}

\bigskip

\

\section{Optimal Control and Weak KAM Theorem}

In this section, we give a proof of Theorem \ref{wKAM} using some ideas from \cite{BeBu1} and \cite{BeBu2}. More precisely, we will prove the following.

\begin{thm}\label{wKAMcts}
Assume that the function $(t,x,y)\mapsto c_t(x,y)$ defined by (\ref{controlcost}) is continuous and the manifold $M$ is compact, then there exists a unique constant $h$ such that the Hamilton-Jacobi-Bellman equation (\ref{sHJB}) has a viscosity solution.
\end{thm}

We start the proof by introducing the Lax-Oleinik semigroup:

\begin{equation}
S_tf(y)=\inf_{x\in M}[c_t(x,y)+f(x)].
\end{equation}

\begin{thm}\label{viscosity}
Assume that the function $(t,x,y)\mapsto c_t(x,y)$ defined by (\ref{controlcost}) is continuous. Then, for each function $f$, the function $(t,x)\mapsto S_tf(x)$ is continuous on $(0,\infty)\times M$. Moreover, it is a viscosity solution to the Hamilton-Jacobi-Bellman equation $\partial_t f+H(x,\partial_x f)=0$ on $(0,\infty)\times M$.
\end{thm}

\begin{proof}
Continuity of the function $S_tf$ follows immediately from that of $c_t$ and compactness of the manifold $M$. The fact that it is a viscosity solution follows as in \cite{Ev}.
\end{proof}

The following theorem is a continuous version of \cite[Lemma 9]{BeBu2} and the proof is similar.

\begin{thm}\label{cpctc}
Assume that the function $(t,x,y)\mapsto c_t(x,y)$ is continuous and the manifold $M$ is compact. Then, for each $a>0$, the family $\{c_t|t\geq a\}$ is equicontinuous. Moreover, there exists constants $h$ and $K$ such that
\[
|c_t(x,y)-ht|\leq K
\]
for all $t\geq a$ and all $x,y$ in $M$.
\end{thm}

\begin{proof}
The function $(t,x,y)\mapsto c_t(x,y)$ is uniformly continuous on $[a,b]\times M\times M$ for some constants $b>2a>0$. So, given $\e>0$, there is a $\delta>0$ such that
\[
|c_\tau(x_1,y_1)-c_\tau(x_2,y_2)|<\e/2
\]
whenever $d(x_1,x_2)<\delta$, $d(y_1,y_2)<\delta$ and $a<\tau<b$.

Assume that $t\geq a$. Since $b>2a$, there is a partition $0=t_0\leq t_1\leq t_2\leq ...\leq t_l=t$ such that $a< t_{i+1}-t_i<b$. Assume that $c_t(x_1,y_1)\geq c_t(x_2,y_2)$ and let $x_2=z_0,z_1,...,z_l=y_2$ be points on the manifold $M$ such that $c_t(x_2,y_2)=\sum_{i=1}^lc_{t_i}(z_{i-1},z_i)$. Then we have
\[
\begin{split}
&c_t(x_1,y_1)-c_t(x_2,y_2)\\&\leq c_{t_1-t_0}(x_1,z_1)-c_{t_1-t_0}(x_2,z_1)+c_{t_l-t_{l-1}}(z_{l-1},y_1)-c_{t_l-t_{l-1}}(z_{l-1},y_2)\\&<\e
\end{split}
\]
whenever $d(x_1,x_2)<\delta$ and $d(y_1,y_2)<\delta$.
It follows that $\{c_t|t\geq a\}$ is an equicontinuous family.

Let $M_t=\sup_{x,y}c_t(x,y)$ and $m_t=\inf_{x,y}c_t(x,y)$, where the supremum and the infimum are taken over all pairs of points of the manifold $M$. Let $t_1$ and $t_2$ be two positive numbers and let $z$ be a point on the manifold $M$ such that $c_{t_1+t_2}(x,y)=c_{t_1}(x,z)+c_{t_2}(z,y)$. It follows from this $M_{t_1+t_2}\leq M_{t_1}+M_{t_2}$. Similarly, $m_t$ satisfies $m_{t_1+t_2}\geq m_{t_1}+m_{t_2}$. It follows that the infimum of the function $\frac{M_t}{t}$ is finite. Indeed, if the infimum of $\frac{M_t}{t}$ is $-\infty$, then so is $\frac{m_t}{t}$. But note that $\frac{m_{kt_0}}{kt_0}\geq \frac{m_{t_0}}{t_0}$ for all positive integer $k$. This gives a contradiction. It follows that $M:=\inf_t\frac{M_t}{t}$ is finite. Given $\e>0$, we find $t_0$ such that $\frac{M_{t_0}}{t_0}<M+\e$. Every $t> t_0$ can be decompose into $t=kt_0+s$, where $t_0\leq s\leq 2t_0$. It follows that
\[
M\leq \frac{M_t}{t}\leq \frac{kM_{t_0}+M_s}{t}= \frac{M_{t_0}}{t_0}\frac{kt_0}{kt_0+s}+\frac{M_s}{t}<(M+\e)\frac{kt_0}{kt_0+s}+\frac{M_s}{t}.
\]
By continuity of the cost $c$, we know that $M_s$ is bounded. It follows that from this and the above inequality that
\[
\lim\limits_{t\to\infty}\frac{M_t}{t}=M.
\]
Similarly, we also have
\[
\lim\limits_{t\to\infty}\frac{m_t}{t}=m.
\]
Finally, it follows from equicontinuity of the family $\{c_t|t\geq a\}$ that $M_t-m_t\leq C$ for some constant $C$ and for all $t\geq a$. Therefore, $h:=M=m$.
\end{proof}

\begin{lem}\label{cpctf}
Assume that the function $(t,x,y)\mapsto c_t(x,y)$ is continuous and the manifold $M$ is compact. Let $f$ be a bounded function, then the family $\mathcal S:=\{S_tf-ht|t\geq a\}$ is uniformly bounded and equicontinuous.
\end{lem}

\begin{proof}
According to Lemma \ref{cpctc}, the family $\{c_t|t\geq a\}$ is equicontinuous. So, for each $\e>0$, there is a $\delta>0$ such that for all $t\geq a$
\[
|c_t(x_1,y_1)-c_t(x_2,y_2)|<\e/2
\]
whenever $d(x_1,x_2)<\delta$ and $d(y_1,y_2)<\delta$.

By definition of $S_tf$, we can find, for each $\e>0$, a point $z_t$ such that

\[
S_tf(x_2)>c_t(z_t,x_2)+f(z_t)-\e/2.
\]

It follows that
\[
S_tf(x_1)-S_tf(x_2)< c_t(z_t,x_1)+f(z_t)-c_t(z_t,x_2)-f(z_t)+\e/2 <\e.
\]
Since the above equation holds for all $\e$ and all $t\geq a$, we conclude that the family $\mathcal S$ is equicontinuous.

Fix a point $x$ in $M$. For each $\e>0$, let $z$ be a point in $M$ such that

\[
c_t(z,x)+f(z)\geq S_tf(x)>c_t(z,x)+f(z)-\e.
\]

Therefore, by Theorem \ref{cpctc}, we have

\[
K+\sup_{y\in M}\{f(y)\}\geq S_tf(x)-ht>-K+\inf_{y\in M}\{f(y)\}-\e
\]
for some constant $K>0$.
We conclude from this that $\mathcal S$ is uniformly bounded.
\end{proof}

Define the function $\bar f$ by

\[
\bar f(x)=\inf_{t\geq a}[S_tf(x)-ht].
\]
It follows from Lemma \ref{cpctf} that $\bar f$ is bounded. The following theorem taken from \cite{Fa2} together with Theorem \ref{viscosity} and Lemma \ref{cpctf} finish the proof of the existence part of Theorem \ref{wKAMcts}. We give a sketch of the proof here.

\begin{thm}\label{wKAMS}
Assume that there exists a constant $h$ such that the family $\mathcal S:=\{S_tf-ht|t\geq a\}$ is uniformly bounded and equicontinuous, then $S_t\bar f-ht$ converges uniformly to a function $\tilde f$. Moreover, it satisfies
\[
S_t\tilde f-ht = \tilde f.
\]
\end{thm}

\begin{proof}
By applying $S_t$ to the definition of $\bar f$, it is not hard to see that $S_t\bar f-ht\geq \bar f$. Since $S_t$ is order preserving, we can apply $S_t$ again to this inequality to shows that $t\mapsto S_t\bar f(x)-ht$ is increasing for each $x$ in $M$. It follows from this and Lemma \ref{cpctf} that $S_t\bar f(x)-ht$ converges uniformly to a continuous function $\tilde f$. We apply once again $S_t$ to the definition of $\tilde f$ and use the continuity of the semigroup $S_t$, we get $S_t\tilde f-kt = \tilde f$.
\end{proof}

Finally, we finish the uniqueness of the constant $h$ as a corollary of Theorem \ref{cpctc}.

\begin{cor}
Assume that the function $(t,x,y)\mapsto c_t(x,y)$ is continuous and the manifold $M$ is compact. Let $h$ be as in Theorem \ref{cpctc} and let $f$ be a function which satisfies $S_tf-kt=f$ for some number $k$, then $k=h$.
\end{cor}

\begin{proof}
For each natural number $n$, let $z_n$ be points in $M$ which satisfies
\[
c_n(z_n,x)+f(z_n)\geq f(x)+kn=S_nf(x)\geq c_n(z_n,x)+f(z_n)-\frac{1}{n}.
\]
Note that the function $f$ is continuous and $\lim\limits_{n\to\infty}\frac{c_n(z_n,x)}{n}=h$. It follows that if we divide the above inequality by $n$ and let $n$ goes to infinity, we get $k=h$ as claimed.
\end{proof}

\

\section{Optimal Transportation and Weak KAM Theorem}

Let $\mu$ and $\nu$ be two Borel probability measures. Consider the cost function defined in (\ref{controlcost}) and the following Monge-Kantorovich problem of optimal transportation:

\begin{equation}\label{MKcost}
C_T(\mu,\nu)=\inf_{\Pi}\int_{M\times M}c_T(x,y)d\Pi(x,y)
\end{equation}
where the infimum is taken over all measures on $M\times M$ with marginals $\mu$ and $\nu$. That is, if $\pi_1,\pi_2:M\times M\to M$ are the projections onto the first and second entries, then $\pi_{1*}\Pi=\mu$ and $\pi_{2*}\Pi=\nu$.

The above problem (\ref{MKcost}) admits a dual version given by
\begin{equation}\label{dual}
\mathcal I_T(\mu,\nu)=\sup\int_Mg(x)d\nu(x)-\int_Mf(x)\mu(x),
\end{equation}
where the supremum is taken over all pairs of functions $(f,g)$ which satisfy $g(y)-f(x)\leq c_T(x,y)$.

The following theorem is the well known result in \cite{Ka}. See also \cite{Vi1,Vi2}.

\begin{thm}\label{Kandual}
Assume that the function $c_T$ is continuous, then the infimum in (\ref{MKcost}) and the supremum in (\ref{dual}) is achieved. Moreover, for any optimal measure $\Pi$ of (\ref{MKcost}) and any pair of functions $(f,g)$ that maximizes (\ref{dual}), we have that $\Pi$ is concentrated on the set $\{(x,y)\in M\times M|g(y)-f(x)=c_T(x,y)\}$ and  $C_T(\mu,\nu)=\mathcal I_T(\mu,\nu)$.
\end{thm}

Note that if $(f,g)$ maximizes (\ref{dual}), then so is $(f,S_Tf)$. We define
\begin{equation}\label{alphaT}
\alpha_T:=\inf_{\mu}\frac{1}{T}C_T(\mu,\mu),
\end{equation}
where the infimum is taken over all Borel probability measures on $M$.

The following lemma can be proved in same way as in \cite[Lemma 33]{BeBu1}.
\begin{lem}\label{costcovex}
There exists a measure $\mu$ which achieves the infimum in (\ref{alphaT}).
\end{lem}

The next theorem is a generalization of a result \cite{BeBu1} which gives another characterization of the number $h$ in Theorem \ref{wKAM}.

\begin{thm}\label{alpha}
Under the assumptions in Theorem \ref{wKAM}, we have $\alpha_T=h$ for each $T>0$.
\end{thm}

Following \cite{BeBu2}, we call measures $\Pi$ on the space $M\times M$ generalized Mather measure if $\pi_{1*}\Pi=\pi_{2*}\Pi$ and
\[
\frac{1}{T}\int_{M\times M}c_T(x,y)d\Pi(x,y)=h.
\]
The following corollary describes the support of the generalized Mather measures.

\begin{cor}
Suppose that we make the same assumptions as in Theorem \ref{wKAM} and let $g$ be a function which satisfies $S_tg=g+ht$. If $\Pi$ is a generalized Mather measure which
satisfies
\[
\frac{1}{T}\int_{M\times M}c_T(x,y)d\Pi(x,y)=h,
\]
then the support of $\Pi$ is contained in the set
\[
\{(x,y)|c_T(x,y)=g(y)-g(x)+hT\}.
\]
\end{cor}

\begin{proof}
Let $\mu=\pi_{1*}\Pi=\pi_{2*}\Pi$. Then
\[
C_T(\mu,\mu)=hT=\int_MS_Tgd\mu-\int_Mgd\mu\leq \mathcal I_T(\mu,\mu)=C_T(\mu,\mu).
\]
It follows that the support of $\Pi$ is contained in
\[
\{(x,y)|c_T(x,y)=S_Tg(y)-g(x)\}.
\]
\end{proof}

\begin{proof}[Proof of Theorem \ref{alpha}]
Let $g$ be a function which satisfies $S_tg=g+ht$ and let $\mu$ be a minimizer corresponding to the minimization problem of $\alpha_T$ in (\ref{alphaT}). It follows from Theorem \ref{Kandual} that
\begin{equation}\label{alpha1}
T\alpha_T=C_T(\mu,\mu)=\mathcal I_T(\mu,\mu)\geq \int_M S_Tgd\mu-\int_Mgd\mu=hT
\end{equation}
for all $T>0$.

For the proof of the following lemma, we follow closely \cite[Lemma 7]{BeBu2}.
\begin{lem}\label{interpolation}
Let $\nu_1$ and $\nu_2$ be two Borel probability measures and let $0\leq s\leq T$, then there exists a Borel probability measure $\nu$ such that
\[
C_T(\nu_1,\nu_2)=C_s(\nu_1,\nu)+C_{T-s}(\nu,\nu_2).
\]
\end{lem}

\begin{proof}
By Theorem \ref{Kandual}, we can find measures $\Pi_1,\Pi_2$ on $M\times M$ such that $C_s(\nu_1,\nu)=\int_{M\times M}c_s(x,y)d\Pi_1(x,y)$ and $C_{T-s}(\nu,\nu_2)=\int_{M\times M}c_{T-s}(x,y)d\Pi_2(x,y)$. By disintegration of measures, there are measures $\mu^y_1$ and $\mu^x_2$ such that  $d\Pi_1(x,y)=d\mu_1^y(x)d\nu(y)$ and $d\Pi_2(x,y)=d\mu_2^x(y)d\nu(x)$. Let $\mu$ be the measure defined by
\[
\int_{M\times M}f(x,y)d\Pi(x,y)=\int_{M\times M\times M}f(x,y)d\mu^z_1(x)d\mu^z_2(y)d\nu(z).
\]
It is not hard to check that the marginals of $\Pi$ are $\nu_1$ and $\nu_2$. Therefore, we get
\[
\begin{split}
C_T(\nu_1,\nu_2)&\leq \int_{M\times M}c_T(x,y)d\Pi(x,y) \\& \leq \int_{M\times M\times M}c_s(x,z)+c_{T-s}(z,y)d\mu^z_1(x)d\mu^z_2(y)d\nu(z) \\& =\int_{M\times M}c_s(x,z)d\Pi_1(x,z)+\int_{M\times M}c_{T-s}(z,y)d\Pi_2(z,y)\\&=C_s(\nu_1,\nu)+C_{T-s}(\nu,\nu_2).
\end{split}
\]

Let $\mathcal P$ be the set of pairs of Borel probability measures $(\nu_1,\nu_2)$ which satisfies the conclusion of the lemma. It is not hard to see that $(\delta_x,\delta_y)$ is contained in $\mathcal P$, where $\delta_x$ is the Dirac mass at $x$.  Indeed, let $x(\cdot):[0,T]\to M$ be an admissible path which satisfy $x(0)=x$, $x(T)=y$ and achieve the infimum in (\ref{controlcost}). Then,

\begin{equation}\label{interpolation1}
\begin{split}
C_s(\delta_x,\delta_{x(s)})+C_{T-s}(\delta_{x(s)},\delta_y)=c_s(x,x(s))+c_{T-s}(x(s),y)\\=c_T(x,y)=C_T(\delta_x,\delta_y).
\end{split}
\end{equation}

To finish the proof, it remains to notice that the set $\mathcal P$ is convex and weak-$*$ closed. Therefore, the result follows from approximation by delta masses.
\end{proof}

Now let $\nu$ be a measure which satisfies $C_{NT}(\nu,\nu)=\alpha_{NT}$. It follows from Lemma \ref{interpolation} that there exists Borel probability measures $\nu=\mu_0,\mu_1,...,\mu_N=\nu$ such that
\[
C_{NT}(\nu,\nu)=\sum_{i=1}^NC_T(\mu_{i-1},\mu_i).
\]

Since $\mathcal I_T$ is convex and so is $C_T$. Therefore,
\[
\frac{1}{NT}C_{NT}(\nu,\nu)\geq \frac{1}{T}C_T(\tilde \mu,\tilde \mu)\geq \alpha_T\geq h,
\]
where $\tilde \mu=\frac 1N\sum_{i=1}^N\mu_i$.

Finally, it follows from Theorem \ref{cpctc} that $\lim\limits_{N\to\infty}\frac{1}{NT}C_{NT}(\nu,\nu)=h$. This finishes the proof.
\end{proof}

\end{document}